\documentclass[a4paper,12pt]{amsart}
\usepackage[ngerman, english]{babel}
\usepackage[utf8]{inputenc}
\usepackage{pdfsync}
\usepackage{verbatim}
\usepackage{amsmath}
\usepackage{amsthm}
\usepackage{amssymb}
 \usepackage{paralist}
\theoremstyle{theorem}
\newtheorem{satz}{Theorem}[section]

  \newtheorem{lemma}[satz]{Lemma}
  \newtheorem{kor}[satz]{Corollary}
  \newtheorem{thmx}{Theorem}

  \newtheorem{prop}[satz]{Proposition}
  
  \theoremstyle{definition}
 \newtheorem{defi}[satz]{Definition}
  \newtheorem{bem}[satz]{Remark}

    {\begin{proof}[Beweis]}
    {\end{proof}}
  \newtheorem{ex}[satz]{Example}

\newcommand{\norm}[1]{\lVert#1\rVert}   
\newcommand{\betrag}[1]{\lvert#1\rvert}
\newcommand{\KK}{\mathrm{KK}}

\newcommand{\K}{\mathrm K}

\newcommand{\NN}{\mathbb N}
\newcommand{\ZZ}{\mathbb Z}

\newcommand{\TT}{\mathbb T}

\newcommand{\lk}{\langle}
\newcommand{\rk}{\rangle}
\newcommand{\id}{\text{id}}

\usepackage{tikz}
\usetikzlibrary{matrix,arrows}
\usepackage{hyperref}
\allowdisplaybreaks

\title[$\K$-theory and homotopies of twists on ample groupoids]{$\K$-theory and homotopies of twists on ample groupoids}
\author{Christian B\"onicke}
\address{School of Mathematics and Statistics, University of Glagow, University Gardens, Glasgow, G12 1RX, UK}
\email{christian.bonicke@glasgow.ac.uk}

\thanks{This research was partially supported by Deutsche Forschungsgemeinschaft (SFB 878).}

\subjclass[2010]{46L80, 46L55}
\keywords{Twisted groupoid $\mathrm{C}^*$-algebra, K-theory}
\begin{document}
\maketitle

\begin{abstract}
This paper investigates the $\mathrm{K}$-theory of twisted group\-oid $\mathrm{C}^*$-algebras. It is shown that a homotopy of twists on an ample groupoid satisfying the Baum-Connes conjecture with coefficients gives rise to an isomorphism between the $\K$-theory groups of the respective twisted  groupoid $\mathrm{C}^*$-algebras. 
The results are also interpreted in an inverse semigroup setting and applied to generalized Renault-Deaconu groupoids and $P$-graph algebras.
\end{abstract}

\section{Introduction}
A key idea in noncommutative geometry is to access the study of topological dynamical systems (for example a discrete group acting on a compact space by homeomorphisms) through a naturally associated noncommutative $\mathrm{C}^*$-algebra. Many of these constructions can be conveniently studied simultaneously in the framework of topological groupoids.
Indeed, large classes of $\mathrm{C}^*$-algebras have been shown to admit groupoid models since Renault's seminal work on groupoid $\mathrm{C}^*$-algebras in \cite{Renault1980}. The range of examples covers crossed products associated to transformation groups, (higher rank) graph-$\mathrm{C}^*$-algebras \cite{KumjianPask00}, the uniform Roe algebra \cite{STY02}, and the $\mathrm{C}^*$-algebras associated to various classes of semigroups \cite{Paterson1999,MR3077884} to name but a few prominent ones.
A slight alteration of the construction takes certain cohomological objects called twists associated with a groupoid $G$ into account. A twist over $G$ is an extension of $G$ by the trivial circle bundle $G^{(0)}\times\TT$.

Twisted groupoid $\mathrm{C}^*$-algebras are useful in the study of untwisted algebras \cite{MR3162248}, but the addition of a twist as input data has much further reaching consequences: Indeed, building on the theory of Feldman-Moore in the von Neumann setting \cite{MR0578656,MR0578730}, Kumjian and Renault showed \cite{MR854149,MR2460017} that whenever a $\mathrm{C}^*$-algebra $A$ admits a Cartan subalgebra, then $A$ is in fact isomorphic to a twisted groupoid $\mathrm{C}^*$-algebra. In other words, (twisted) groupoid $\mathrm{C}^*$-algebras appear intrinsically in $\mathrm{C}^*$-algebra theory.

This article is concerned with the $\K$-theory of twisted group\-oid $\mathrm{C}^*$-algebras. The main result establishes that the $\mathrm{K}$-theory of a twisted groupoid $\mathrm{C}^*$-algebra is invariant with respect to homotopies of twists for a large class of groupoids and reads as follows:
\begin{thmx}\label{MainTheoremIntro}(see Theorem \ref{Thm:Evalutation Induces Isomorphism on K-theory})
	Let $G$ be a second countable ample groupoid, which satisfies the Baum-Connes conjecture with coefficients and let $\Sigma$ be a  homotopy of twists for $G$. Then the canonical map $q_t:C_r^*(G\times [0,1];\Sigma)\rightarrow C_r^*(G,\Sigma_t)$ given by evaluation induces an isomorphism $$(q_{t})_*:\K_*(C_r^*(G\times[0,1];\Sigma))\rightarrow \K_*(C_r^*(G,\Sigma_t)).$$
\end{thmx}

Some remarks on Theorem \ref{MainTheoremIntro} are in order: 
First, the result is non-trivial since a homotopy of twists does not result in a homotopy-equivalence of the twisted groupoid $\mathrm{C}^*$-algebras. A basic example is provided by the irrational rotation algebra $\mathcal{A}_\theta$, which can be realized as a twisted group algebra $\mathcal{A}_\theta=C^*(\ZZ^2,\omega_\theta)$ with respect to the $2$-cocycle $\omega_\theta((k,l),(n,m))=e^{2\pi i \theta ln}$. This cocycle is homotopic to the constant cocycle but $\mathcal{A}_\theta$ is not homotopy-equivalent to $C(\TT^2)$ since this would require the existence of a $\ast$-homomorphism $\mathcal{A}_\theta\rightarrow C(\TT^2)$ of which there are none since $\mathcal{A}_\theta$ is a simple, non-trivial noncommutative algebra.
Our second remark concerns the assumptions of Theorem \ref{MainTheoremIntro}:
Although the Baum-Connes conjecture with coefficients is known to be false in general \cite{MR1911663}, it has been confirmed for large classes of groupoids, including all amenable groupoids \cite{Tu98}. Thus, our result applies to large classes of examples.

Many concrete twists arising in examples of interest are actually homotopic to the trivial twist.
In these cases, our result says that the $\K$-theory of the twisted groupoid $\mathrm{C}^*$-algebra does not depend on the twist at all.
Examples include the already mentioned irrational rotation algebras $\mathcal{A}_\theta$ and their higher-dimensional analogues and the Heegaard-type quantum spheres \cite[Example~3.5]{MR3011251}.

The motivation for this work is twofold: There is a strong interest in understanding the $\mathrm{K}$-theory of twisted groupoid $\mathrm{C}^*$-algebras coming from the classification program for simple, unital, nuclear $\mathrm{C}^*$-algebras on the one hand and coming from the physics of quasicrystals on the other.

By the work of numerous researchers (see \cite{MR3583354} and the references therein), there is now a complete classification of separable, simple, unital $\mathrm{C}^*$-algebras of finite nuclear dimension satisfying Rosenberg and Schochet's universal coefficient theorem (UCT).
The classifying invariant is the so called Elliott-invariant, whose main ingredient is operator $\K$-theory.
Combining the astonishing results from \cite{1802.01190,MR2329002}, every $\mathrm{C}^*$-algebra falling within the scope of classification admits a twisted groupoid model. Even better, the nuclearity assumption in the classification theorems implies that the underlying groupoid is amenable and hence satisfies the Baum-Connes conjecture.
This is closely linked to the question whether all nuclear $\mathrm{C}^*$-algebras satisfy the UCT and to the existence of Cartan subalgebras \cite{MR3672919}.

In another direction, twisted groupoid $\mathrm{C}^*$-algebras and their $\K$-theo\-ry play an important role in physics: In the article \cite{MR2119241} and the references therein a link to classification of $D$-brane charges in string theory is explained.
More directly related to our work is another important application studied in \cite{MR3438328}: A twisted version of Bellissard's Gap Labelling Conjecture \cite{MR1798994}.
The classical (untwisted) version of the conjecture is a statement about the possible gaps in the spectrum of certain Schrödinger operators which arises in solid state physics. Mathematically, the conjecture is about a description of the image of a canonical trace map on the $\K_0$-group of a certain groupoid $\mathrm{C}^*$-algebra. A twist over that groupoid enters the picture once one allows for the presence of a non-zero magnetic field.

In conclusion, the $\K$-theory of twisted groupoid $\mathrm{C}^*$-algebras plays a major role in the classification program for simple $\mathrm{C}^*$-algebras as well as in the applications from physics. This provides a strong motivation to study methods that enable the computation of $\K$-theory for twisted groupoid $\mathrm{C}^*$-algebras.
This question has been addressed before in the case of groups by Echterhoff, Lück, Phillips and Walters in \cite{MR2608195} and Gillaspy treated the cases of transformation groups, higher-rank graphs and group bundles in a series of articles \cite{MR3407179,MR3346133,MR3563179}. Using the machinery developed by the author in \cite{1806.00391} this article presents a unified approach and considerable extension of the above mentioned results.

To make the result accessible for readers with different backgrounds we also interpret Theorem \ref{MainTheoremIntro} in the setting of twisted actions of inverse semigroups in the sense of \cite{MR2821242,MR1620499}.
Finally, we apply our results to semidirect product groupoids arising from (partial) actions of subsemigroups of groups. These groupoids can be viewed as a generalisation of the transformation groupoids associated to ordinary group actions. In particular, we can recover and generalize the main results of \cite{MR3407179}.

The article is organized as follows: 
After recalling some basic facts about étale groupoids, groupoid dynamical systems and the associated reduced crossed products, we turn to the study of twists. We show in Proposition \ref{Prop:LocalSections}, that every twist over an étale groupoid locally admits continuous cross sections and conclude that every twist over an ample groupoid admits a global continuous cross section.
In section \ref{Section:MainSection} we prove our main result. The idea is to use the Going-Down principle to reduce to the case of compact groupoids, where the result can be proven directly.
In the two final sections we look at some examples: First, we interpret our main result in terms of twisted actions of inverse semigroups, and second, we study certain semidirect-product groupoids arising from actions of subsemigroups of groups.

\section{Preliminaries}
\subsection{Étale and ample groupoids}
Recall, that a \textit{groupoid} is a set $G$ together with a subset $G^{(2)}\subseteq G\times G$, called the set of \textit{composable pairs,} a \textit{product} map $(g,h)\mapsto gh$ from $G^{(2)}$ to $G$ and an \textit{inverse} map $g\mapsto g^{-1}$ from $G$ onto $G$, such that the following hold:
	\begin{enumerate}
		\item The product is associative: If $(g_1,g_2),(g_2,g_3)\in G^{(2)}$ for some $g_1,g_2,g_3\in G$, then we also have $(g_1g_2,g_3),(g_1,g_2g_3)\in G^{(2)}$ and 
		\begin{equation*}
		(g_1g_2)g_3=g_1(g_2g_3).
		\end{equation*}
		\item The inverse map is involutive, i.e. $(g^{-1})^{-1}=g$ for all $g\in G$.
		\item $(g,g^{-1})\in G^{(2)}$ for all $g\in G$ and if $(g,h)\in G^{(2)}$, then 
		\begin{equation*}
		g^{-1}(gh)=h \ \textup{ and }\ (gh)h^{-1}=g.
		\end{equation*}
	\end{enumerate}
The fact that multiplication is partially defined implies that multiple elements may act as (partial) units:
	The set $G^{(0)}:=\lbrace g\in G\mid g=g^{-1}=g^2\rbrace$ is called the \textit{set of units} in $G$. There are canonical maps $d:G\rightarrow G^{(0)}$ given by $d(g)=g^{-1}g$ and $r:G\rightarrow G^{(0)}$ given by $r(g)=gg^{-1}$, called the \textit{domain} and \textit{range} \textit{map} respectively.

For subsets $A,B\in G^{(0)}$ we will write $G_A:=d^{-1}(A)$, $G^B:=r^{-1}(B)$ and $G_A^B:=G_A\cap G^B$. If $A$ (and/or $B$) consists just of a single unit $u\in G^{(0)}$ we will omit the braces (e.g.: we will write $G^u:=r^{-1}(\lbrace u\rbrace)$).

In this article we will be concerned with topological groupoids:
We say that $G$ is a locally compact Hausdorff groupoid, if $G$ is a groupoid, which is equipped with a locally compact Hausdorff topology, such that the multiplication and inversion map are continuous. The fact that $G$ is Hausdorff ensures that the unit space $G^{(0)}$ is closed in $G$.

A locally compact groupoid is called \textit{étale}, if $d:G\rightarrow G$ is a local homeomorphism, i.e. every point $g\in G$ has an open neighbourhood $U\subseteq G$, such that $d(U)$ is open in $G$ and $d_{\mid U}:U\rightarrow d(U)$ is a homeomorphism.
It follows easily from the definition that for an étale groupoid $G$
the unit space $G^{(0)}$ is open in $G$ and for each $u\in G^{(0)}$ the sets $G^u$ and $G_u$ are discrete (in the subspace topology).

The results of the present article apply to a large subclass of the étale groupoids:
An étale groupoid $G$ is called \textit{ample}, if its space of units $G^{(0)}$ is totally disconnected.
There is a wealth of groupoids studied in the literature that belong to this class, e.g groupoids associated to aperiodic tilings and quasicrystals (see \cite{MR2658985}), groupoids associated to directed graphs (see \cite{MR1432596}) and higher-rank graphs (see \cite{KumjianPask00}), groupoids associated to inverse semigroups (see \cite{Paterson1999}), and the coarse groupoid studied in large scale geometry (see \cite{STY02}).

\subsection{Groupoid crossed products}
In this short paragraph we remind the reader of the definitions of groupoid dynamical systems and the associated reduced crossed products.
Let $G$ be an étale groupoid.
A \textit{groupoid dynamical system} $(A,G,\alpha)$ consists of a $C_0(G^{(0)})$-algebra $A$ and a family $(\alpha_g)_{g\in G}$ of $*$-isomorphisms $\alpha_g:A_{d(g)}\rightarrow A_{r(g)}$ such that $\alpha_{gh}=\alpha_g\circ \alpha_h$ for all $(g,h)\in G^{(2)}$ and such that $g\cdot a:=\alpha_g(a)$ defines a continuous action of $G$ on the upper-semicontinuous bundle $\mathcal{A}$ over $G^{(0)}$ associated with $A$.

Associated with this data is a $\mathrm{C}^*$-algebra called the reduced crossed product. We will briefly run through the constructions roughly following \cite{MR1900993}.
Given a groupoid dynamical system $(A,G,\alpha)$ as above, consider the complex vector space $\Gamma_c(G,r^*\mathcal{A})$. It carries a canonical $*$-algebra structure with respect to the following operations:
$$(f_1\ast f_2)(g)=\sum\limits_{h\in G^{r(g)}} f_1(h)\alpha_h(f_2(h^{-1}g))$$
and
$$f^*(g)=\alpha_g(f(g^{-1})^*).$$
See for example \cite[Proposition~4.4]{MR2547343} for a proof of this fact.
For $u\in G^{(0)}$ consider the Hilbert $A_u$-module $\ell^2(G^u,A_u)$. It is the completion of the space of finitely supported $A_u$-valued functions on $G^u$, with respect to the inner product 
$$\lk \xi,\eta\rk=\sum\limits_{h\in G^u}\xi(h)^*\eta(h).$$
We can then define a $*$-representation $\pi_u:\Gamma_c(G,r^*\mathcal{A})\rightarrow \mathcal{L}(\ell^2(G^u,A_u))$ by
$$\pi_u(f)\xi(g)=\sum\limits_{h\in G^u}\alpha_g(f(g^{-1}h))\xi(h).$$
Using this family of representations, we can define a $\mathrm{C}^*$-norm on the convolution algebra $\Gamma_c(G,r^*\mathcal{A})$ by
$$\norm{f}_r:=\sup\limits_{u\in G^{(0)}}\norm{\pi_u(f)}.$$
The reduced crossed product $A\rtimes_r G$ is defined to be the completion of $\Gamma_c(G,r^*\mathcal{A})$ with respect to $\norm{\cdot}_r$.

\section{Twists over ample groupoids}
Let us recall the definition of a twist over a groupoid:
	\begin{defi}
		Let $G$ be a topological groupoid. A \textit{twist} $\Sigma$ over $G$ is a central groupoid extension
		$$G^{(0)}\times\TT\stackrel{i}{\longrightarrow} \Sigma\stackrel{j}{\longrightarrow} G,$$
		by which we mean:
		\begin{enumerate}
			\item The map $i$ is a homeomorphism onto $j^{-1}(G^{(0)})\subseteq \Sigma$,
			\item the map $j$ is a continuous and open surjection, and
			\item the extension is central meaning that $i(r(\sigma),z)\sigma=\sigma i(d(\sigma),z)$ for all $\sigma\in \Sigma$ and $z\in\TT$.
		\end{enumerate}
		We say that $\Sigma$ is a \textit{continuous twist} over $G$, if $j$ admits a continuous cross section.
	\end{defi}
	Note, that we will canonically identify $\Sigma^{(0)}$ with $G^{(0)}$ and for all $u\in G^{(0)}$ we have $j(i(u,z))=u$. Moreover, $\Sigma$ admits a canonical left action of $\TT$ given by $z\cdot \sigma:=i(r(\sigma),z)\sigma$. Hence a twist can also be viewed as a principal $\TT$-bundle over $G$.
	
	\begin{bem}\label{Rem:Cocycles}
		Twists over groupoids are closely related to $2$-cocycles on them.
		Recall, that a $2$-cocycle for $G$ is a map $\omega:G^{(2)}\rightarrow \TT$, such that $$\omega(g_1,g_2)\omega(g_1g_2,g_3)=\omega(g_1,g_2g_3)\omega(g_2,g_3)$$
		for all $g_1,g_2,g_3\in G$ with $(g_1,g_2),(g_2,g_3)\in G^{(2)}$. It is called \textit{normalized} if in addition one has
		$$\omega(g,d(g))=1=\omega(r(g),g)$$
		for all $g\in G$.
		
		Given such a $2$-cocycle $\omega$ on $G$ we can define a groupoid structure on $\Sigma_\omega:=G\times  \TT$ as follows: Two elements $(g_1,s_1),(g_2,s_2)\in \Sigma_\omega$ are composable if $(g_1,g_2)\in G^{(2)}$ and their product is defined as
		$$(g_1,s_1)(g_2,s_2):=(g_1g_2,s_1s_2\omega(g_1,g_2)).$$
		The inverse of $(g,s)\in \Sigma_\omega$ is given by
		$$(g,s)^{-1}:=(g^{-1},\overline{s\omega(g^{-1},g)}).$$
		If $\omega$ is continuous, it is not hard to check that $\Sigma_\omega$ is a locally compact Hausdorff groupoid in the product topology. Thus, we obtain a central extension of groupoids
		$$ G^{(0)}\times\TT \stackrel{i}{\longrightarrow} \Sigma_\omega\stackrel{j}{\longrightarrow} G,$$
		where the first map is the canonical inclusion and the second map is the projection onto the first factor. Note, that $j$ has a canonical continuous cross section $s$ given by $s(g)=(g,1)$.
		
		Conversely, starting with a twist $\Sigma$ over $G$ admitting a continuous section $s:G\rightarrow \Sigma$, we note that $j(s(gh)^{-1}s(g)s(h))\in G^{(0)}$. Hence, by exactness we get $s(gh)^{-1}s(g)s(h)\in i(G^{(0)}\times \TT)$. Since $i$ is a homeomorphism onto its image, we obtain a continuous map $\omega:G^{(2)}\rightarrow\TT$ by letting $\omega(g,h)=i^{-1}(s(gh)^{-1}s(g)s(h))$. It is then routine to check that $\omega$ satisfies the cocycle identity. It is normalized, provided that the section $s$ restricts to the identity on $G^{(0)}$.
		If $G$ is an étale Hausdorff groupoid, this can always be arranged: If $s:G\rightarrow \Sigma$ is any continuous section, then define a new section $s':G\rightarrow \Sigma$ by $s'(g)=s(g)$ for all $g\in G\setminus G^{(0)}$ and $s'(u)=u$ for all $u\in G^{(0)}$. Then $s'$ is still continuous since $G^{(0)}$ is clopen in $G$.
		
		We remark, that it is a well-known fact that twists do not generally admit a continuous cross section and thus are more general than $2$-cocycles (see \cite{MR1174207}).
	\end{bem}

	As a first goal we wish to show that twists over ample groupoids are automatically continuous. To this end we need to introduce some technology developed in \cite{VanErp}.
	Let $X$ be a locally compact Hausdorff space and $E$ be a Hilbert $C_0(X)$-module. Write $\mathcal{E}$ for the associated bundle of Hilbert spaces $\mathcal{E}=\coprod_{x\in X} E_x$, which can be equipped with a topology turning it into a continuous Hilbert bundle over $X$ such that $E$ is isomorphic to the continuous sections of the bundle $\mathcal{E}\rightarrow X$ vanishing at infinity denoted by $\Gamma_0(X,\mathcal{E})$.
	Associated with such a module is the groupoid
	$$Iso(E):=\lbrace (x,V,y)\mid x,y\in X\textit{ and }V:E_y\rightarrow E_x\textit{ is a unitary}\rbrace,$$
	equipped with the obvious structure:
	$$(x,V,y)(y,W,z):=(x,VW,z)\textit{ and }(x,V,y)^{-1}:=(y,V^*,x).$$
	It was shown in \cite[Proposition~3.3]{VanErp} that $Iso(E)$ can be equipped with a canonical Hausdorff topology making it into a topological group\-oid.
	
	In a similar spirit one can define the automorphism groupoid of an upper semicontinuous $\mathrm{C}^*$-bundle $p:\mathcal{A}\rightarrow X$ over $X$ by
	$$Aut(\mathcal{A}):=\lbrace (x,\alpha,y)\mid x,y\in X,\ \alpha:A_y\rightarrow A_x \textit{ is a }\ast-isomorphism\rbrace.$$
	Again, $Aut(\mathcal{A})$ can be equipped with a Hausdorff topology making it into a topological groupoid by \cite[Proposition~3.1]{VanErp}.
	
	The two constructions presented above are closely related. If $E$ is a Hilbert $C_0(X)$-module as above, then we can consider the $\mathrm{C}^*$-algebra $K(E)$ of compact operators on $E$. This algebra is a $C_0(X)$-algebra in a canonical way with fibres $K(E)_x=K(E_x)$. Consequently, we can form the associated upper semicontinuous $\mathrm{C}^*$-bundle denoted by $\mathcal{K}\rightarrow X$, such that $\Gamma_0(X,\mathcal{K})\cong K(E)$ as $C_0(X)$-algebras.
	Every unitary operator $V:E_y\rightarrow E_x$ clearly defines a $\ast$-isomorphism $\mathrm{Ad}\ V:K(E_y)\rightarrow K(E_x)$ given by $(\mathrm{Ad}\ V)(T)=VTV^*$. Hence we obtain a canonical groupoid homomorphism $\mathrm{Ad}:Iso(E)\rightarrow Aut(\mathcal{K})$. It was shown in \cite[Proposition~3.4]{VanErp} that $\mathrm{Ad}$ fits into a short exact sequence of topological groupoids
	$$X\times \TT\stackrel{i}{\rightarrow} Iso(E)\stackrel{\mathrm{Ad}}{\rightarrow} Aut(\mathcal{K}),$$
	such that $\mathrm{Ad}$ is a continuous and open surjection, and $i$ is a homeomorphism onto the kernel of $\mathrm{Ad}$.
	 Moreover, the map $\mathrm{Ad}$ admits local continuous sections, in the sense that for every $(x,\alpha,y)\in Aut(\mathcal{K})$ there exists a neighbourhood $N$ of $(x,\alpha,y)\in Aut(\mathcal{K})$ and a continuous map $\beta:N\rightarrow Iso(E)$ such that $\mathrm{Ad} \circ\beta=\id_N$.
	
	We want to apply this machinery to twists over groupoids. Although much of what follows also works for more general groupoids we restrict ourselves to the class of étale groupoids at this point to avoid unnecessary technicalities. So let $\Sigma$ be a twist over an étale groupoid $G$. 
	Consider the complex vector space
	$$C_c(G;\Sigma):=\lbrace f\in C_c(\Sigma)\mid f(z\sigma)=zf(\sigma)\rbrace.$$ 
	Then we obtain a Hilbert $C_0(G^{(0)})$-module $E$ by separation and completion of $C_c(G;\Sigma)$ with respect to the inner product
	\begin{equation}\label{Eq:TwistedInnerProduct}
	\lk f_1,f_2\rk_{C_0(G^{(0)})}(u)=\sum\limits_{j(\sigma)\in G_u} \overline{f_1(\sigma)}f_2(\sigma).
	\end{equation}
		Each of the fibres $E_u$ is a Hilbert space and the following lemma gives a convenient description of those.
		\begin{lemma}\label{Lem:ContinuousTwist}
			If $G^{(0)}\times \TT\rightarrow\Sigma\stackrel{j}{\rightarrow}G$ is a twist over $G$ then the fibre $E_u$ over $u\in G^{(0)}$ can be identified with the Hilbert space obtained by completion of $E_0(u)=\lbrace f\in C_c(\Sigma_u)\mid f(z\sigma)=zf(\sigma)\rbrace$ with respect to the inner product $\lk f_1,f_2\rk=\sum_{j(\sigma)\in G_u} \overline{f_1(\sigma)}f_2(\sigma).$
		\end{lemma}
		\begin{proof}
			One easily sees that the restriction map $C_c(G;\Sigma)\rightarrow E_0(u)$ factors through an isometric linear map $E_u\rightarrow \overline{E_0(u)}$. The only issue is the surjectivity of the restriction map.
			To this end let $f_0\in C_c(\Sigma)$ be an extension of $f$.
			Then $f_1:\Sigma\rightarrow \TT$ defined by
			$$f_1(\sigma):=\int_{\TT}\overline{z}f_0(z\sigma)dz,$$
			is an element of $E_0$, that extends $f$. 
			
		\end{proof}
		Let $\rho:\Sigma\rightarrow Iso(E)$ be the representation given by $(\rho(\sigma)\xi)(\tau)=\xi(\tau\sigma)$. Then $\rho(z\sigma)=z\rho(\sigma)$ for all $z\in \TT$ and $\sigma\in \Sigma$. Consequently, denoting the upper semicontinuous $\mathrm{C}^*$-bundle associated to the $C_0(G^{(0)})$-algebra $K(E)$ by $\mathcal{K}$ again, we obtain a well-defined continuous groupoid homomorphism $\alpha:G\rightarrow Aut(\mathcal{K})$ by $\alpha_{j(\sigma)}=\mathrm{Ad}\ \rho(\sigma)$.
		This homomorphism fits into the following commutative diagram:
		\begin{center}
			\begin{tikzpicture}[description/.style={fill=white,inner sep=2pt}]
			\matrix (m) [matrix of math nodes, row sep=3em,
			column sep=2.5em, text height=1.5ex, text depth=0.25ex]
			{ G^{(0)}\times\TT & \Sigma & G \\
				G^{(0)}\times \TT & Iso(E) & Aut(\mathcal{K}) \\
			};
			\path[->,font=\scriptsize]
			(m-1-1) edge node[auto] {$  $} (m-1-2)
			
			(m-2-1) edge node[auto] {$  $} (m-2-2)
			
			(m-1-1) edge node[auto] {$ \id $} (m-2-1)
			(m-1-2) edge node[auto] {$ \rho $} (m-2-2)
			(m-1-2) edge node[auto] {$ j $} (m-1-3)
						(m-2-2) edge node[auto] {$ \mathrm{Ad} $} (m-2-3)
						(m-1-3) edge node[auto] {$ \alpha $} (m-2-3)
			;
			\end{tikzpicture}
		\end{center}	
	
	The following result seems to be folklore but we could not locate it in the literature, so we provide a proof.
	\begin{prop}\label{Prop:LocalSections}
		Let $G^{(0)}\times\TT\stackrel{i}{\longrightarrow} \Sigma\stackrel{j}{\longrightarrow} G$ be a twist over an étale groupoid $G$. Then $j$ admits local cross sections in the sense that for every $g\in G$ there exists a neighbourhood $U\subseteq G$ of $g$ and a continuous map $s:U\rightarrow \Sigma$ such that $j(s(h))=h$ for all $h\in U$.
	\end{prop}
	\begin{proof}
		Let $E$ be the Hilbert $C_0(G^{(0)})$-module associated with $\Sigma$ and $Iso(E)$ and $Aut(\mathcal{K})$ as above. Since the canonical map $\mathrm{Ad}:Iso(E)\rightarrow Aut(\mathcal{K})$ has local cross sections, so does every pullback along this map. Hence it suffices to show that $\Sigma$ is isomorphic to the pullback groupoid $$\Sigma'=\lbrace (V,g)\mid g\in G, V:E_{d(g)}\rightarrow E_{r(g)}\textit{ unitary with } \mathrm{Ad}\ V=\alpha_g\rbrace.$$ 
		
		First of all, there exists a canonical groupoid homomorphism $\varphi:\Sigma\rightarrow \Sigma'$ given by $\varphi(\sigma):=(\rho(\sigma),j(\sigma))$.
		For the injectivity assume that $\varphi(\sigma)=\varphi(\sigma')$. Then $j(\sigma)=j(\sigma')$, so using exactness we get that there exists a $z\in\TT$ such that $\sigma'=z\sigma$. From the definition of $\rho$ we get that for any $f\in C_c(G;\Sigma)$ and any $\tau\in \Sigma$ such that $d(\tau)=r(\sigma)=r(\sigma')$ we have $f(\tau\sigma)=f(\tau\sigma')$. In particular, we get $f(\sigma)=f(\sigma')=zf(\sigma)$ for all $f\in C_c(G;\Sigma)$. We conclude that $z=1$ and hence $\sigma=\sigma'$ as desired.
		
		For the surjectivity let $(V,g)\in \Sigma'$. First choose $\sigma_0\in \Sigma$ with $j(\sigma_0)=g$. Then $\mathrm{Ad}\ \rho(\sigma_0)=\alpha_{j(\sigma_0)}=\alpha_g=\mathrm{Ad}\ V$ and hence $\rho(\sigma_0)=z V$ for some $z\in\TT$. Now put $\sigma:=\overline{z}\sigma_0$. Then $j(\sigma)=g$ and $\rho(\sigma)=\overline{z}\rho(\sigma_0)=V$ and hence $\varphi(\sigma)=(V,g)$ as desired.
		
		Since $\varphi$ is clearly continuous it remains to prove that it is also open. To see this we will employ \cite[Proposition~1.15]{Williams}. So let $(V_\lambda,g_\lambda)_\lambda$ be a net converging to $\varphi(\sigma)=(\rho(\sigma),j(\sigma))$ in $\Sigma'$. In particular we get $g_\lambda\rightarrow j(\sigma)$ in $G$. Since $j$ is an open surjection we can pass to a subnet and relabel to assume that there exists a net $(\sigma_\lambda)$ in $\Sigma$ such that $j(\sigma_\lambda)=g_\lambda$ and $\sigma_\lambda\rightarrow \sigma$.
		Moreover, $\mathrm{Ad}\ \rho(\sigma_\lambda)=\alpha_{j(\sigma_\lambda)}=\mathrm{Ad}\ V_\lambda$ and consequently $\rho(\sigma_\lambda)=z_\lambda V_\lambda$ for a suitable net $(z_\lambda)_\lambda$ in $\TT$. Now since $\TT$ is compact, we can pass to yet another subnet (and relabel) to assume that the net $(z_\lambda)_\lambda$ converges to some $z\in \TT$. Since $\rho(\sigma)=\lim_\lambda \rho(\sigma_\lambda)=\lim_\lambda z_\lambda V_\lambda=z\rho(\sigma)$, we get $z=1$.
		To sum up we found a (sub)net of $(V_\lambda,g_\lambda)$ and a net $(\overline{z_\lambda}\sigma_\lambda)_\lambda$ in $\Sigma$ converging to $\sigma$ such that $\varphi(\overline{z}_\lambda\sigma_\lambda)=(\rho(\overline{z}_\lambda\sigma_\lambda),j(\sigma_\lambda))=(V_\lambda,g_\lambda)$.
	\end{proof}
	
	\begin{kor}\label{Cor:TwistsOnAmpleGroupoidsAreContinuous}
	Every twist $\Sigma$ over a $\sigma$-compact ample groupoid $G$ is continuous.
	\end{kor}
	\begin{proof}
		By the previous proposition the map $j:\Sigma\rightarrow G$ admits local cross sections. Combining this with the assumptions that $G$ is $\sigma$-compact and totally disconnected, we can find a countable covering of $G=\bigcup U_n$ by compact open subsets which are the domains of local sections of $j$. Using a standard inclusion-exclusion argument we can recursively define a partition $G=\bigsqcup_{n\in\NN} V_n$ into compact open subsets $V_n$, which are the domains of local sections $s_n:V_n\rightarrow \Sigma$ of $j$. Then we can piece all of these sections together in the obvious way to obtain a global continuous section $s:G\rightarrow \Sigma$ of $j$ as desired.
	\end{proof}
	
	Following \cite{MR1174207}, we associate a $\mathrm{C}^*$-algebra to a twist $\Sigma$ over an étale groupoid $G$ as follows:
	Consider the complex vector space
	$$C_c(G;\Sigma):=\lbrace f\in C_c(\Sigma)\mid f(z\sigma)=zf(\sigma)\rbrace.$$ 
	Then $C_c(G;\Sigma)$ becomes a $\ast$-algebra with repect to the operations
	$$f_1\ast f_2(\sigma)=\sum\limits_{j(\tau)\in G^{r(\sigma)}} f_1(\tau)f_2(\tau^{-1}\sigma)\text{ and }f^*(\sigma)=\overline{f(\sigma^{-1})}.$$
	Observe that the sum makes sense, since the expression $f_1(\tau)f_2(\tau^{-1}\sigma)$ only depends on $j(\tau)\in G$. 
	For each $u\in G^{(0)}$ let $E_u$ be the fibre of the Hilbert $C_0(G^{(0)})$-module $E$ associated with $\Sigma$ as above.
	Then, for $f\in C_c(G;\Sigma)$ we can define an operator $\pi_u(f)$ on $E_u$ by 	$\pi_u(f)\xi)=f\ast\xi$. The operator $\pi_u(f)$ is bounded and we define $C_r^*(G;\Sigma)$ to be the completion of $C_c(G;\Sigma)$ with respect to the norm
	$$\norm{f}_r:=\sup\limits_{u\in G^{(0)}} \norm{\pi_u(f)}.$$

	\begin{bem}
		In the literature one often finds a direct construction of the twisted groupoid $\mathrm{C}^*$-algebra associated to a continuous $2$-cocycle $\omega$, that does not pass through the canonical extension $\Sigma_\omega$ explained above. It is defined as a completion of the convolution algebra $C_c(G)$ with product and involution given by
		$$f_1\ast_\omega f_2(g)=\sum\limits_{h\in G^{r(g)}} f_1(h)f_2(h^{-1}g)\omega(h,h^{-1}g)\text{ and }$$
		$$f^*(g)=\overline{f(g^{-1})\omega(g,g^{-1})},$$
		and we will denote it by $C_r^*(G,\omega)$.
		Note, that both constructions yield the same $\mathrm{C}^*$-algebras, since there is a canonical isomorphism $\Phi:C_r^*(G,\Sigma_\omega)\rightarrow C_r^*(G,\omega)$, given by $\Phi(f)(g)=f(g,1)$. One can easily define an inverse map $\Psi:C_r^*(G,\omega)\rightarrow C_r^*(G,\Sigma_\omega)$ by $\Psi(f)(g,z)=zf(g)$.
	\end{bem}

	We shall need the following version of the Packer-Raeburn stabilisation trick in the groupoid framework.
	\begin{prop}\label{Prop:PackerRaeburn}\cite[Proposition~5.1]{VanErp}
		Let $\Sigma$ be a twist over an étale groupoid $G$. If $E$ is the associated Hilbert $C_0(G^{(0)})$-module and $\alpha:G\rightarrow Aut(\mathcal{K})$ is the homomorphism constructed above, then $\alpha$ defines a groupoid dynamical system $(K(E),G,\alpha)$ such that $K(E)\rtimes_{\alpha,r} G$ is Morita equivalent to $C_r^*(G;\Sigma)$.
	\end{prop}
	For later reference, we also need to briefly recall the construction of the bimodule that provides the Morita equivalence:	
	Let $A_0$ be the dense subalgebra $\Gamma_c(G,r^*\mathcal{K})\subseteq K(E)\rtimes_{r,\alpha} G$.
	Following \cite[Theorem~6.4]{MR2446021} together with the formulas given in the proof of \cite[Proposition~5.1]{VanErp} one then defines a pre-Hilbert bimodule-structure on $X_0:=\Gamma_c(G,d^*\mathcal{E})$ as follows: For $\xi,\eta\in X_0$ and $f\in A_0$ define
	$$(f\xi)(g)=\sum\limits_{h\in G^{r(g)}} \alpha_{g^{-1}h}(f(h^{-1}))\xi(h^{-1}g),$$
	
	$${}_{A_0}\lk \xi,\eta\rk(g)=\sum\limits_{h\in G^{s(g)}} \alpha_{gh}({}_{K(E_x)}\lk\xi(h),\eta(gh)\rk).$$
	Note that in \cite{VanErp}, the authors construct the crossed product by completing $\Gamma_c(G,s^*\mathcal{K})$. Thus, in order to obtain the formulas above we need to pass through the canonical isomorphism, sending $f\in A_0$ to the function $\check{f}\in \Gamma_c(G,s^*\mathcal{K})$, given by $\check{f}(g):=f(g^{-1})$.
	
	For $\xi,\eta\in X_0$ and $f\in C_c(G;\Sigma)$ define
	$$(\xi f)(g)=\sum\limits_{j(\sigma)\in G^{d(g)}}f(\sigma^{-1})\rho(\sigma)\xi(gj(\sigma))$$
	and a $C_c(G;\Sigma)$-valued inner product by
	$$\lk \xi,\eta\rk (\tau)=\sum\limits_{j(\sigma)\in G^{d(\tau)}}\lk \rho(\sigma^{-1}\tau^{-1})\xi(j(\sigma^{-1}\tau^{-1})),\rho(\sigma^{-1})\eta(j(\sigma^{-1}))\rk (d(\sigma))$$
	The completion $X$ of $X_0$ then implements a Morita equivalence between $K(E)\rtimes_r G$ and $C_r^*(G;\Sigma)$. With this description at hand we can prove the following technical little lemma, which will turn out useful later:
	\begin{lemma}\label{Lemma:PR-Inductive Limit}
	If $(\xi_i)_i$ is a net in $X_0$ converging to $\xi\in X_0$ in the inductive limit topology, then $\norm{\xi_i-\xi}\rightarrow 0$.
	\end{lemma}
	\begin{proof}
	Let $\eta_i:=\xi-\xi_i\in X_0$. We will show, that ${}_{A_0}\lk \eta_i,\eta_i\rk$ converges to zero in the inductive limit topology. Then it will also converge to zero in the reduced norm and hence $\norm{\xi-\xi_i}^2=\norm{{}_{A_0}\lk \eta_i,\eta_i\rk}\rightarrow0$ as desired.
	By assumption, there exists a compact subset $K\subseteq G$ such that $supp(\eta_i)\subseteq K$ for all $i$. Since the action of $G$ on itself by multiplication is always proper, the set $C:=\lbrace g\in G\mid g^{-1}K\cap K\neq\emptyset\rbrace$ is also compact. Now if $0\neq {}_{A_0}\lk \eta_i,\eta_i\rk(g)=\sum_{h\in G^{d(g)}}\alpha_{gh}(\lk \eta_i(h),\eta_i(gh)\rk)$, there exists some $h\in G^{d(g)}$ such that $\lk \eta_i(h),\eta_i(gh)\rk\neq 0$.
	But then necessarily $h\in g^{-1}K\cap K$, which implies $g\in C$. Thus $supp({}_{A_0}\lk \eta_i,\eta_i\rk)\subseteq C$ for all $i$.
	Now let $\varepsilon>0$ be given. Choose $M>0$ such that $\sup_{u\in G^{(0)}}\betrag{K^u}\leq M$. Then we have $\sup_{g\in G}\norm{\lk \eta_i(g),\eta_i(g)\rk}<\frac{\sqrt{\varepsilon}}{M}$ for $i$ large enough.
	For $i$ large enough we can then compute 
	\begin{align*}
	\norm{\lk \eta_i,\eta_i\rk(g)}&\leq \sum\limits_{h\in G^{d(g)}}\norm{\lk \eta_i(h),\eta_i(gh)}\\
	&\leq \sum\limits_{h\in G^{d(g)}}\norm{\lk \eta_i(h),\eta_i(h)}\norm{\lk \eta_i(gh),\eta_i(gh)}<\varepsilon,
	\end{align*}
	which finishes the proof.
	\end{proof}
	\begin{bem}\label{Rem:PackerRaeburnRestrict}
		Let $\Sigma$ be a twist over the étale groupoid $G$ and $H\subseteq G$ a compact open subgroupoid. Then $\Sigma':=j^{-1}(H)$ is easily seen to be a twist over $H$.
		Let $E$ and $\alpha$ be as above. Then we can restrict the action $\alpha$ to an action of $H$ on $K(E)_{\mid H}$. We claim that the resulting crossed product $K(E)_{\mid H}\rtimes_r H$ is then Morita equivalent to $C_r^*(H;\Sigma')$. The proof is basically the same as in \cite[Proposition~5.1]{VanErp}, we just restrict all the appearing bundles to the subgroupoid $H$ and use the fact that $H$ is an $(H,H)$-equivalence.
	\end{bem}
	\section{Homotopies of twists}\label{Section:MainSection}
	Our goal is to prove that the $\K$-theory of $C_r^*(G;\Sigma)$ only depends on the homotopy class of $\Sigma$. We will start by formalizing what we mean by a homotopy:
	Given a locally compact Hausdorff groupoid $G$, consider the trivial bundle of groupoids $G\times[0,1]$ with the product topology. This bundle is itself a locally compact groupoid, where $(g,s)$ and $(h,t)$ are composable if $g$ and $h$ are composable in $G$ and $s=t$. In this case we define their product by $(g,s)(h,s):=(gh,s)$ and an inverse by $(g,s)^{-1}:=(g^{-1},s)$. Consequently, the unit space is given by $G^{(0)}\times [0,1]$.
	
	\begin{defi}
		A (continuous) twist $\Sigma$ over $G\times[0,1]$ is called a (continuous) \textit{homotopy of twists} for $G$.
	\end{defi}
	If $\Sigma$ is a homotopy of twists over $G$ then $\Sigma$ is a continuous field of groupoids over $[0,1]$ in the sense of \cite[Definition~8.9]{Delaroche} since for all $\sigma\in\Sigma$ we have $pr_{[0,1]}(d(\sigma))=pr_{[0,1]}(r(\sigma))$.
	 In particular, for each $t\in[0,1]$ we obtain a twist $\Sigma_t$ over $G$ by letting $\Sigma_t:=(pr_{[0,1]}\circ r)^{-1}(\Sigma)=\Sigma_{\mid G^{(0)}\times \lbrace t\rbrace}$.
	 
	If $G$ is ample, then each $\Sigma_t$ is a continuous twist by Corollary \ref{Cor:TwistsOnAmpleGroupoidsAreContinuous}, so one may wonder if $\Sigma$ is always continuous as a twist over $G\times [0,1]$. Although $G\times [0,1]$ is no longer ample, this can be arranged by a standard patching argument:
	
	\begin{prop}\label{Prop:HomotopiesAreContinuous}
		Let $G$ be a $\sigma$-compact ample groupoid and $\Sigma$ be a homotopy of twists for $G$. Then $\Sigma$ is automatically a continuous homotopy of twists, in the sense that there exists a continuous section $s:G\times [0,1]\rightarrow \Sigma$.
	\end{prop}
	\begin{proof}
		We start to prove a preliminary claim:
		Suppose $V\subseteq G$ is an open subset and $0\leq a<b<c\leq 1$ are such that there exist continuous sections $s_1:V\times [a,b]\rightarrow \Sigma$ and $s_2:V\times [b,c]\rightarrow \Sigma$. Then we claim that there exists a continuous section $s:V\times [a,c]\rightarrow \Sigma$.
		Indeed, since $s_1$ and $s_2$ are continuous sections, we can define a continuous map $f:V\rightarrow \TT$ by $f(g)=s_1(g,b)s_2(g,b)^{-1}$. Then $\tilde{s}_2(g,t):=f(g)s_2(g,t)$ is still a continuous section, but it has the virtue that $\tilde{s}_2(g,b)=s_1(g,b)$ for all $g\in V$. Hence we can piece $s_1$ and $\tilde{s}_2$ together to obtain the desired function $s:V\times [a,c]\rightarrow \Sigma$.
		
		We now return to the proof of the proposition: Fix $g\in G$ for the moment. For each $t\in[0,1]$ we can apply Proposition \ref{Prop:LocalSections} to find a compact open neighbourhood $V_t$ of $g$ and some open interval $U_t\subseteq [0,1]$ and a continuous section $s_t:V_t\times U_t\rightarrow \Sigma$. Using compactness of $[0,1]$ and intersecting the resulting finite number of neighbourhoods among the $V_t$ we find a finite sequence of numbers $0=t_0<t_1<\ldots <t_n=1$, a compact open neighbourhood $V$ of $g$, and finitely many continuous sections $s_i:V\times [t_{i-1},t_i]\rightarrow \Sigma$. Now we successively apply the first paragraph of this proof to obtain a continuous section $s_g:V\times [0,1]\rightarrow \Sigma$.
		Finally, we use $\sigma$-compactness to proceed as in the proof of Corollary \ref{Cor:TwistsOnAmpleGroupoidsAreContinuous} and piece these sections together to a globally continuous cross section $s:G\times [0,1]\rightarrow \Sigma$.
	\end{proof}
	
	For every $t\in[0,1]$ we obtain a canonical $\ast$-homo\-morphism 
	$$q_t:C_r^*(G\times[0,1];\Sigma)\rightarrow C_r^*(G;\Sigma_t),$$
	which for $f\in C_c(G\times[0,1];\Sigma)$ is
	given by $q_t(f)=f_{\mid \Sigma_t}$. An argument very similar to the proof of Lemma \ref{Lem:ContinuousTwist} shows, that $q_t$ is surjective.
	The main goal of this article is to prove the following result:
	\begin{satz}\label{Thm:Evalutation Induces Isomorphism on K-theory}
		Let $G$ be an ample groupoid, which satisfies the Baum-Connes conjecture with coefficients and let $\Sigma$ be a homotopy of twists for $G$ and $t\in[0,1]$. Then $$(q_{t})_*:\K_*(C_r^*(G\times[0,1];\Sigma))\rightarrow \K_*(C_r^*(G,\Sigma_t))$$
		is an isomorphism.
	\end{satz}
	
	The following result settles the case of continuous homotopies of twists over compact groupoids.
	\begin{prop}\cite[Proposition~3.1]{MR3346133}\label{Prop:CompactCase} If $\Sigma$ is a continuous homotopy of twists on a compact Hausdorff groupoid $G$, then the canonical $\ast$-homomorphism $q_t:C_r^*(G\times[0,1];\Sigma)\rightarrow C_r^*(G;\Sigma_t)$ is a homotopy equivalence.
	\end{prop}

	The rough idea in proving Theorem \ref{Thm:Evalutation Induces Isomorphism on K-theory} is to use the Going-Down principle to reduce the question to the case of compact groupoids and then appeal to Proposition \ref{Prop:CompactCase} above.
	Before proceeding to the precise form of the Going-Down principle we are going to use, the reader may wish to recall the definitions of Le Gall's $\KK^G$-theory \cite{LeGall99}, the topological $\K$-theory $\K_*^{\mathrm{top}}(G;A)$ of a groupoid with coefficients in a $G$-algebra $A$, and the Baum-Connes assembly map $\mu_A:\K_*^{\mathrm{top}}(G;A)\rightarrow \K_*(A\rtimes_r G)$ \cite{MR1798599}. See also \cite{Bonicke,1806.00391} for a detailed overview.
	\begin{satz}\cite[Theorem~7.10]{1806.00391}
		Let $G$ be an ample, second countable, locally compact Hausdorff groupoid and let $A$ and $B$ be separable $G$-algebras. Suppose there is an element $x\in \KK^G(A,B)$ such that 
		\[ \KK^H(C(H^{(0)}),A_{\mid H})\stackrel{\cdot\otimes res_H^G(x)}{\rightarrow}\KK^H(C(H^{(0)}), B_{\mid H})\]
		is an isomorphism for all compact open subgroupoids $H\subseteq G$. Then the Kasparov-product with $x$ induces an isomorphism
		$$ \cdot\otimes x:\K_*^{\mathrm{top}}(G;A)\rightarrow \K_*^{\mathrm{top}}(G;B).$$
	\end{satz}
	
	From now on fix a homotopy of twists $\Sigma$ over an étale Hausdorff groupoid $G$.
	Consider the canonical Hilbert $C_0(G^{(0)}\times [0,1])$-module $E$, defined as the completion of $C_c(G\times[0,1];\Sigma)$ with respect to the inner product defined in equation \ref{Eq:TwistedInnerProduct}.
	Now by Proposition \ref{Prop:PackerRaeburn} and the discussion thereafter we obtain an action $\alpha$ of $G\times[0,1]$ on $K(E)$. Observe, that there is a canonical action of $G$ on $G^{(0)}\times[0,1]$ given by $g\cdot (d(g),t)=(r(g),t)$, such that $G\times[0,1]\cong G\ltimes(G^{(0)}\times [0,1])$. Taking this point of view we can use the pushforward construction (see for example \cite[Proposition~3.10]{1806.00391}) to obtain an action $\beta$ of $G$ on $K(E)$. One has the following:
	\begin{prop}\cite[Theorem~3.8]{1703.05190}
	The canonical map
	$\Phi:\Gamma_c(G\times[0,1],r^*\mathcal{K})\rightarrow \Gamma_c(G,r^*\mathcal{K})$ given by
	$\Phi(f)(g)(t)=f(g,t)$
	is a $\ast$-homo\-morphism and extends to a $\ast$-isomorphism $$\Phi:K(E)\rtimes_{\alpha,r} (G\times[0,1])\rightarrow K(E)\rtimes_{\beta,r} G.$$
	\end{prop}
	On the other hand for each $t\in[0,1]$ we can apply (the proof of) Proposition \ref{Prop:PackerRaeburn} to the twist $\Sigma_t$ over $G$, in order to obtain a Hilbert $C_0(G^{(0)})$-module $E_t$ and an action $\alpha_t$ of $G$ on $K(E_t)$. Let us record the following easy observations concerning the relationship between $E$ and $E_t$:
	\begin{lemma}
		The restriction map $C_c(G\times[0,1];\Sigma)\rightarrow C_c(G;\Sigma_t)$, $f\mapsto f_{\mid \Sigma_t}$ extends to a surjective bounded linear map $p_t:E\rightarrow E_t$.
	\end{lemma}
	\begin{proof}
	It is routine to check, that the restriction map is bounded and linear. Using an argument similar to the proof of Lemma \ref{Lem:ContinuousTwist} one sees that the restriction map $C_c(G\times[0,1];\Sigma)\rightarrow C_c(G;\Sigma_t)$ is surjective. This is not quite enough to conclude that $p_t$ is surjective. However, if $i_t:G^{(0)}\rightarrow G^{(0)}\times[0,1]$ denotes the inclusion at $t\in[0,1]$, then $p_t$ factors through an isometric linear map $i_t^*E\rightarrow E_t$. Since this map is isometric, it is enough to know that the dense subset $C_c(G;\Sigma_t)$ is contained in the image to conclude surjectivity. Using, that the canonical map $E\cong E\otimes_{C_0(G^{(0)}\times[0,1])}C_0(G^{(0)}\times[0,1])\rightarrow i_t^*E$ is surjective, the result follows.
	\end{proof}
	We can use an argument similar to the proof of Lemma \ref{Lem:ContinuousTwist} again, to show that for $u\in G^{(0)}$ we can canonically identify the Hilbert spaces $E_{(u,t)}$ and $(E_t)_u$ and hence also $K(E)_{(u,t)}$ with $K(E_t)_u$.
	Let $X$ and $X_t$ be the equivalence bimodules obtained from applying Proposition \ref{Prop:PackerRaeburn} to the twists $\Sigma$ and $\Sigma_t$, respectively. We have the following:
	\begin{prop}\label{Prop:Bimodules}
		The canonical restriction map $\Gamma_c(G\times[0,1],d^*\mathcal{E})\rightarrow \Gamma_c(G,d^*\mathcal{E}_t)$, $\xi\mapsto \xi_{\mid G\times\lbrace t\rbrace}$ extends to a bounded linear map $\Psi_t:X\rightarrow X_t$ and factors through an isomorphism
		$$\Theta_t:q_t^*(X)\rightarrow X_t$$
		of Hilbert $C_r^*(G;\Sigma_t)$-modules.
	\end{prop}
	\begin{proof}
		From the definition of the respective inner products it is obvious that $\lk \Psi_t(\xi),\Psi_t(\eta)\rk=q_t(\lk \xi, \eta\rk)$ for all $\xi,\eta\in \Gamma_c(G\times[0,1],d^*\mathcal{E})$. It follows that $\Psi_t$ is bounded and hence extends to all of $X$. Define $\Theta_t:q_t^*X=X\otimes_{q_t}C_r^*(G;\Sigma_t)\rightarrow X_t$ on elementary tensors by $\Theta_t(\xi\otimes a)=\Psi_t(\xi)a$. Then $\Theta_t$ extends to an isometric map on all of $q_t^*X$, since
		for $\xi,\eta\in X$ and $a,b\in C_r^*(G;\Sigma_t)$ we can compute
		\begin{align*}
			\lk \xi\otimes a,\eta\otimes b\rk & = (q_t(\lk \eta,\xi\rk)a)^*b\\
			& = (\lk \Psi_t(\eta),\Psi_t(\xi)\rk a)^*b\\
			& = \lk \Psi_t(\xi)a, \Psi_t(\eta) b\rk\\
			& = \lk \Theta(\xi\otimes a),\Theta(\eta\otimes b)\rk.
		\end{align*}
		Finally, to see that $\Theta_t$ is surjective, it is enough to show that it has dense image. First, let $\xi\in \Gamma_c(G,d^*\mathcal{E}_t)$ be of the form $\xi=\varphi\otimes e$ with $\varphi\in C_c(G)$ and $e\in E_t$, i.e. $\xi(g)=\varphi(g)e(d(g))$. Since $p_t:E\rightarrow E_t$ is surjective, we can find an element $e'\in E$ such that $p_t(e')=e$. Also, pick any map $\varphi'\in C_c(G\times[0,1])$, such that $\varphi'(g,t)=\varphi(g)$. Then $\varphi'\otimes e'\in X$ such that $\Psi_t(\varphi'\otimes e')=\xi$. Now if $\xi\in \Gamma_c(G,d^*\mathcal{E}_t)$ is arbitrary we can approximate it in the inductive limit topology by finite sums of elements of the form $\varphi\otimes e$ as above.
		An application of Lemma \ref{Lemma:PR-Inductive Limit} completes the proof.
	\end{proof}
	Let $x\in \KK(K(E)\rtimes_r G,C_r^*(G\times[0,1];\Sigma))$ and $x_t\in \KK(K(E_t)\rtimes_r G,C_r^*(G;\Sigma_t))$ be the canonical KK-equivalences associated to the equivalence bimodules $X$ and $X_t$ respectively. 
	\begin{lemma}\label{Lem:Evaluation1}
	For each $t\in[0,1]$ restriction of functions induces a $G$-equi\-variant $\ast$-homo\-morphism $\Phi_t:K(E)\rightarrow K(E_t)$, such that the following diagram commutes:\begin{center}
			\begin{tikzpicture}[description/.style={fill=white,inner sep=2pt}]
			\matrix (m) [matrix of math nodes, row sep=3em,
			column sep=2.5em, text height=1.5ex, text depth=0.25ex]
			{ \K_*(K(E)\rtimes_r G) & \K_*(K(E_t)\rtimes_{\alpha_t,r}G) \\
				\K_*(C_r^*(G\times[0,1];\Sigma)) & \K_*(C_r^*(G;\Sigma_t)) \\
			};
			\path[->,font=\scriptsize]
			(m-1-1) edge node[auto] {$ (\Phi_t\rtimes G)_* $} (m-1-2)
			
			(m-2-1) edge node[auto] {$ (q_t)_* $} (m-2-2)
			
			(m-1-1) edge node[auto] {$ \cdot\otimes x $} (m-2-1)
			(m-1-2) edge node[auto] {$ \cdot\otimes x_t $} (m-2-2)
			
			;
			\end{tikzpicture}
		\end{center}	
	\end{lemma}
	
	\begin{proof}
	Recall that $K(E)$ is a $C_0(G^{(0)}\times[0,1])$-algebra. Let $\mathcal{K}(E)$ denote the associated bundle. Similarly $K(E_t)$ is a $C_0(G^{(0)})$-algebra with associate bundle $\mathcal{K}(E_t)$. For $f\in \Gamma_0(G^{(0)}\times[0,1],\mathcal{K}(E))=K(E)$ and $u\in G^{(0)}$ define
	$$\Phi_t(f)(u):=f(u,t).$$
	Then $\Phi_t(f)\in \Gamma_0(G^{(0)},\mathcal{K}(E_t))\cong K(E_t)$ and it is straightforward to verify, that $\Phi_t$ is a $G$-equivariant $\ast$-homomorphism.
	To see commutativity of the diagram, it is enough to check that $[\Phi_t\rtimes G]\otimes x_t=x\otimes [q_t]$ in $\KK(K(E)\rtimes_r G,C_r^*(G;\Sigma_t))$. Since all the elements involved can be represented by Kasparov-triples, where the operator is zero, these products are easy to describe: The element on the left handside can be represented by the triple $(X_t,\Phi_t\rtimes G,0)$, while the right handside is given by the class of $(X\otimes_{q_t}C_r^*(G;\Sigma_t),\psi\otimes 1,0)$, where $\psi$ is the left action of $K(E)\rtimes_r G$ on $X$.
	From Proposition \ref{Prop:Bimodules} we have an isomorphism of  right Hilbert $C_r^*(G;\Sigma_t)$-bimodules
	$\Theta:X\otimes_{q_t}C_r^*(G;\Sigma_t)\rightarrow X_t$ given on elementary tensors by $\xi\otimes a\mapsto \xi_{\mid G\times\lbrace t\rbrace}a$.
	Thus, to complete the proof, we observe that $\Theta$ intertwines the left actions of $K(E)\rtimes_r G$.
	\end{proof}
	
	\begin{proof}[Proof of Theorem \ref{Thm:Evalutation Induces Isomorphism on K-theory}]
	Fix $t\in[0,1]$. In light of Lemma \ref{Lem:Evaluation1} and the fact that $G$ satisfies the Baum-Connes conjecture with coefficients (and the naturality of the Baum-Connes assembly map),
	it is enough to show that $\Phi_t:K(E)\rightarrow K(E_t)$ induces an isomorphism
	$$(\Phi_t)_*:\K_*^{\mathrm{top}}(G;K(E))\rightarrow \K_*^{\mathrm{top}}(G;K(E_t)).$$
	Hence we are in the position to apply \cite[Theorem~7.10]{1806.00391} to deduce that it is enough to show that
	$$(\Phi_t\rtimes H)_*:\K_*(K(E)_{\mid H}\rtimes H)\rightarrow \K_*(E_t)_{\mid H}\rtimes H)$$
	is an isomorphism for all compact open subgroupoids $H\subseteq G$.
	Using Remark \ref{Rem:PackerRaeburnRestrict} and the same arguments as in Lemma \ref{Lem:Evaluation1} for $H$, we conclude that it is enough to prove that 
	$$(q_t)_*:\K_*(C_r^*(H\times[0,1];\Sigma))\rightarrow \K_*(C_r^*(H;\Sigma_t))$$
	is an isomorphism.
	Since $G$ is ample, $\Sigma$ is a continuous homotopy of twists by Proposition \ref{Prop:HomotopiesAreContinuous}. Clearly, the restriction of $\Sigma$ to $H$ has the same property. Consequently, we may apply Proposition \ref{Prop:CompactCase} to finish the proof.
	\end{proof}

	\section{Twisted actions of inverse semigroups}
	In this section we phrase our results in terms of twisted actions of inverse semigroups. These have been studied by several authors and we will recall the definition in the setting we want to work in for the readers convenience. In what follows we will write $E$ for the lattice of idempotents in an inverse semigroup $S$.
	For convenience, we will also assume that there is an element $0\in S$ such that $0s=s0=0$ for all $s\in S$. If a semigroup does not have a zero element, one can adjoin a zero element and extend the multiplication in the obvious way.
	\begin{defi}\label{Def:TwistedAction}
	A \textit{twisted action} $\theta$ of an inverse semigroup $S$ on a locally compact Hausdorff space $X$ is a triple $$\theta=(\lbrace D_s\rbrace_{s\in S},\lbrace \theta_s\rbrace_{s\in S},\lbrace \omega_{s,t}\rbrace_{s,t\in S}),$$ consisting of a family of open subsets $D_s$ of $X$ whose union covers $X$, a family of homeomorphisms $\theta_s:D_{s}\rightarrow D_{s^*}$, and a family of continuous maps $\omega(s,t):D_{st}\rightarrow \TT$, such that the following hold:
	\begin{enumerate}
	\item The pair $(\lbrace D_s\rbrace_{s\in S},\lbrace \theta_s\rbrace_{s\in S})$ defines an action of $S$ on $X$, i.e. $\theta_r\circ \theta_s=\theta_{rs}$;
	\item $\omega(s,t)(\theta_r^{-1}(x))\omega(r,st)(x)=\omega(r,s)(x)\omega(rs,t)(x)$, whenever $x\in D_r\cap D_{st}$; and
	\item $\omega(s,e)=1_{se}$ and $\omega(e,s)=1_{es}$ for all $s\in S$ and $e\in E$.
	\end{enumerate}
	\end{defi}

	It is is routine to check that a twisted action of $S$ on $X$ as defined above induces a twisted action of $S$ on $C_0(X)$ in the sense of  \cite{MR1620499} and \cite{MR2821242}.
	Following \cite[Theorem~7.2]{MR2821242}, for any such twisted action we can construct an étale groupoid $G$ with unit space $X$ and a twist $\Sigma$ over $G$ such that the twisted groupoid $\mathrm{C}^*$-algebra $C_r^*(G;\Sigma)$ is isomorphic to the twisted crossed product $C_0(X)\rtimes_{\theta,\omega}^r S$.
	It follows from \cite[Proposition~7.4]{MR2821242} that twisted actions in the sense above yield groupoids with continuous twists since we insist on condition $(3)$ above. One should note that Buss and Exel study a more general version of twisted actions. But since our main result applies to ample groupoids (and hence to continuous twists), we can stick to the definition given above without loosing any generality. More precisely, we have the following interpretation of Corollary \ref{Cor:TwistsOnAmpleGroupoidsAreContinuous} in the setting of inverse semigroups:
	\begin{kor} Let $S$ be an inverse semigroup and $X$ be a totally disconnected locally compact Hausdorff space.
	Then every twisted action $\theta$ of $S$ on $X$ in the sense of \cite[Definition~4.1]{MR2821242} is a twisted action in the sense of Definition \ref{Def:TwistedAction}.
	\end{kor}
	Another obstruction to apply our main results is the requirement that the groupoid in question must be Hausdorff. The groupoid associated to a twisted action of an inverse semigroup is constructed as a certain groupoid of germs. The question when such a groupoid is Hausdorff is a subtle one. It has been studied by Steinberg and was refined by Exel and Pardo. In order to state their result recall, that there is a canonical order relation on an inverse semigroup $S$ given by 
	$$s\leq t\Leftrightarrow s=ts^*s \text{ for all } s,t\in S.$$
	\begin{satz}\cite[Theorem~3.15]{MR3448414}
		Let $(\lbrace D_s\rbrace_{s\in S},\lbrace \theta_s\rbrace_{s\in S})$ be an action of the inverse semigroup $S$ on a locally compact Hausdorff space $X$ and let $G(\theta)$ denote the associated groupoid of germs. Then $G(\theta)$ is Hausdorff if and only if for every $s\in S$ the union $\bigcup_{\lbrace e\in E:e\leq s\rbrace} D_e$ is closed relative to $D_{s^*s}$.		
	\end{satz}
	In particular, if $S$ is a \textit{weak semilattice} then the groupoid of germs $G(\theta)$ is Hausdorff for any action $\theta$ of $S$ on a locally compact Hausdorff space $X$ such that $D_e$ is clopen for every $e\in E$ (see \cite[Theorem~5.17]{MR2565546}).
	
	The previous discussion applies in particular to twisted versions of the canonical action of $S$ on the spectrum of its idempotents:
	Recall, that a \textit{character} on the idempotent semilattice $E$ of $S$ is a non-zero map $\chi:E\rightarrow \lbrace 0,1\rbrace$, such that $\chi(ef)=\chi(e)\chi(f)$ and $\chi(0)=0$. The set of all characters
	$$\widehat{E}=\lbrace \chi:E\rightarrow \lbrace 0,1\rbrace\mid \chi\textit{ is a character}\rbrace$$
	is called the spectrum of $E$. Equipped with the topology inherited from the product topology on $\lbrace 0,1\rbrace^E$, the spectrum $\widehat{E}$ is a locally compact, totally disconnected Hausdorff space.
	There is a canonical action of $S$ on $\widehat{E}$ given as follows:
	For $s\in S$ let $D_s:=\lbrace \chi\in \widehat{E}\mid \chi(s^*s)=1\rbrace$ and $\theta_s:D_s\rightarrow D_{s^*}$ be given by $\theta_s(\chi)(e):=\chi(s^*es)$. Then each $D_s$ is clopen in $\widehat{E}$ and $\theta$ is an action of $S$ on $\widehat{E}$.
	The groupoid of germs of this action is usually called the \textit{universal groupoid} of $S$ and denoted by $G(S)$ (see \cite{Paterson1999,zbMATH05553518} for more details on this construction).
	
	By \cite{MR2821242} again, twisted actions of $S$ on $\widehat{E}$ are in one-to-one correspondence with continuous twists over the universal groupoid $G(S)$ of $S$. A twisted action of $S$ on $\widehat{E}$ whose underlying action is the canonical one described above (i.e. the data mainly consists of a family of maps $\omega_{s,t}:D_{st}\rightarrow \TT$ satisfying the appropriate conditions from Definition \ref{Def:TwistedAction}) will also be called a \textit{twist} over $S$ or a \textit{$2$-cocycle} on $S$.

	In light of the above discussion there is an obvious notion of homotopy of twisted actions:
	\begin{defi}
	A \textit{homotopy of twisted actions} of $S$ on $X$ is a triple $$(\lbrace D_s\rbrace_{s\in S},\lbrace \theta_s\rbrace_{s\in S},\lbrace \omega_{r,s}\rbrace_{r,s\in S})$$ consisting of a family of open subsets $D_s$ of $X$ whose union covers $X$, a family of homeomorphisms $\theta_s:D_{s}\rightarrow D_{s^*}$, and a family of continuous maps $\omega(r,s):D_{rs}\times [0,1]\rightarrow \mathbb{T}$, such that
	$(\lbrace D_s\times [0,1]\rbrace_{s\in S},\lbrace\theta_s\times\id_{[0,1]}\rbrace_{s\in S},\lbrace\omega(r,s)\rbrace_{r,s\in S})$ is a twisted action of $S$ on $X\times [0,1]$.
	\end{defi}
	Note that a homotopy defines a family of twisted actions of $S$ on $X$, indexed by the unit interval $[0,1]$.
	
	Invoking \cite[Theorem~7.2]{MR2821242} again, a homotopy of twisted actions provides us with an étale groupoid $\mathcal{G}$ with $\mathcal{G}^{(0)}=X\times[0,1]$ and a twist $\Sigma$ over $\mathcal{G}$ such that $C_r^*(\mathcal{G};\Sigma)\cong C_0(X\times[0,1])\rtimes_{\theta,\omega}^r S$.
	Since in our definition of homotopy the action of $S$ on $X$ does not depend on $t\in [0,1]$, the groupoid $\mathcal{G}$ is just the trivial bundle of groupoids $\mathcal{G}=G\times [0,1]$, where $G$ is the groupoid of germs associated to the action $(\lbrace D_s\rbrace_{s\in S},\lbrace \theta_s\rbrace_{s\in S})$ of $S$ on $X$.
	Hence the following result is a direct consequence of Theorem \ref{Thm:Evalutation Induces Isomorphism on K-theory} and \cite{Tu98}:
	\begin{satz}
	Let $(\lbrace D_s\rbrace_{s\in S},\lbrace \theta_s\rbrace_{s\in S},\lbrace \omega_{s,t}\rbrace_{s,t\in S})$ be a homotopy of twisted actions of $S$ on a locally compact totally disconnected Hausdorff space $X$. Suppose further, that
	\begin{enumerate}
		\item for every $s\in S$ the union $\bigcup_{\lbrace e\in E:e\leq s\rbrace} D_e$ is closed relative to $D_{s^*s}$, and
		\item the action $\theta$ of $S$ on $X$ is amenable in the sense of \cite[Definition~3.2]{MR3614034}.
	\end{enumerate}  Then the evaluation homomorphism $$q_t:C_0(X\times[0,1])\rtimes_{\theta\times \id,\omega}^r S\rightarrow C_0(X)\rtimes_{\theta,\omega_t}^r S$$
	induces an isomorphism in $\K$-theory.
	\end{satz}
	Specialising to the canonical action of $S$ on $\widehat{E}$ again we get the following:
	For a homotopy of $2$-cocycles $\omega$ on $S$ and $t\in[0,1]$ let $\omega^t$ be the $2$-cocycle on $S$ given by $\omega^t_{r,s}(\chi):=\omega_{r,s}(\chi,t)$.
	Given a $2$-cocycle $\omega=(\omega_{s,t})_{s,t\in S}$ on an inverse semigroup $S$ let $C_r^*(S,\omega):=C_0(\widehat{E})\rtimes_{\theta,\omega}^r S$.
	Then we have the following immediate corollary:
	\begin{kor}
		Let $S$ be an inverse semigroup. Suppose that $S$ is a weak semilattice and that the canonical action of $S$ on $\widehat{E}$ is amenable.
		If $\omega$ is a homotopy of $2$-cocycles on $S$, then
		$$\K_*(C_r^*(S,\omega^0))\cong \K_*(C_r^*(S,\omega^1)).$$
	\end{kor}

	 \section{Applications}
	 \subsection{Generalized Renault-Deaconu groupoids}
	 Our results apply nicely to actions of certain semigroups studied in the literature. Recall the following definition from \cite{MR3592511}:
	 \begin{defi} Let $P$ be a subsemigroup of a discrete group $\Gamma$ containing the identity element $e$ of $\Gamma$ and let $X$ be a locally compact Hausdorff space. Then a (continuous) \textit{left} action of $P$ consists of a family of local homeomorphisms $(T_p)_{p\in P}$ with $T_p:dom(p)\rightarrow ran(p)$ such that 
	 	\begin{enumerate}
	 		\item $dom(p),ran(p)$ are open subsets of $X$ for all $p\in P$,
	 		\item $dom(e)=ran(e)=X$ and $T_e=\id_X$, and
	 		\item for all $p,q\in P$ we have $x\in dom(pq)$ if and only if $x\in dom(q)$ and $T_q(x)\in dom(p)$ and in that case $T_{pq}(x)=(T_p\circ T_q)(x)$.
	 	\end{enumerate}
	 Moreover, such an action is called \textit{directed}, provided that for every pair of elements $p,q\in P$ with $dom(p)\cap dom(q)\neq \emptyset$ there exists an upper bound $r\in P$ of $p$ and $q$ such that $dom(p)\cap dom(q)\subseteq dom(r)$.
	 \end{defi}
	 
	 For example, every action of an Ore semigroup is directed, and the action of a quasi-latticed ordered semigroup on its spectrum is directed.
	 
	 Given a directed action $T$ of $P$ on $X$, where $P$ is a subsemigroup of the discrete group $\Gamma$ one can define the \textit{semidirect product groupoid} (or generalized Renault-Deaconu groupoid) $G(X,P,T)$ of the action as the set
	 $$\lbrace (x,\gamma,y)\in X\times\Gamma\times X\mid \exists p,q\in P: \gamma=p^{-1}q\textit{ and }T_p(x)=T_q(y)\rbrace,$$
	 and equip it with the groupoid structure it inherits as a subset of $X\times\Gamma\times X$, equipped with the obvious groupoid structure over $X$.
	 It can be shown that the sets
	 $$Z(U,p,q,V)=\lbrace (x,p^{-1}q,y)\in G(X,P,T)\mid x\in U, y\in V\rbrace,$$
	with $U,V$ ranging over the open subsets of $X$ and $p,q\in P$ form a basis for a locally compact Hausdorff topology on $G(X,P,T)$ giving it the structure of an étale groupoid over $X$, such that the canonical homomorphism $c:G(X,P,T)\rightarrow \Gamma$ given by $c(x,\gamma,y)=\gamma$ is continuous (see \cite[Proposition~5.12]{MR3592511} for the details). In particular, if $X$ happens to be totally disconnected, then $G(X,P,T)$ is an ample groupoid.
	Moreover, if the ambient group $\Gamma$ is amenable then the groupoid $G(X,P,T)$ is topologically amenable by \cite[Theorem~5.13]{MR3592511}.
	Hence we can apply our main result to twists over such groupoids.
	
	The following lemma illustrates how twists on such a semidirect product groupoid can arise in practice. We omit its routine proof.
	\begin{lemma}
		Let $(X,P,T)$ be a directed action of $P\subseteq \Gamma$ on a locally compact Hausdorff space $X$.
		If $\omega:\Gamma^2\rightarrow \TT$ is a (normalized) $2$-cocycle for $\Gamma$, then $$\tilde{\omega}:G(X,P,T)^{(2)}\rightarrow \TT; \ \ \ \tilde{\omega}((x,\gamma,y),(y,\gamma',z)):=\omega(\gamma,\gamma')$$ defines a continuous (normalized) $2$-cocycle on $G(X,P,T)$.
		Moreover, if $\omega_1,\omega_2:\Gamma ^2\rightarrow \TT$ are cohomologous then so are $\tilde{\omega_1},\tilde{\omega_2}$.
	\end{lemma}
	Similarly, homotopies of $2$-cocycles on $\Gamma$ induce homotopies of continuous $2$-cocycles on $G(X,P,T)$.
	
	\begin{ex}
			Consider the case $(P,\Gamma)=(\NN^k,\ZZ^k)$. It is a well-known fact that every class in $H^2(\ZZ^k,\TT)$ has a representative of the form $\omega_\Theta(n,m)=e^{i\langle \Theta n,m\rangle}$ for a skew symmetric real $k\times k$-matrix $\Theta$. Each of these cocycles is clearly homotopic to the constant $2$-cocycle $\omega_1$.
			It follows that the associated twists $\Sigma_{\Theta}$ (via the previous Lemma and Remark \ref{Rem:Cocycles}) over $G(X,\NN^k,T)$ are all homotopic to the trivial twist and hence
			$\K_*(C_r^*(G(X;\NN^k,T);\Sigma_{\Theta}))\cong \K_*(C_r^*(G(X,\NN^k,T)))$.
		\end{ex}

	We remark that certainly not all twists on $G(X,P,T)$ arise in this way, unless $X$ is zero-dimensional. To see this consider the ordinary Renault-Deaconu groupoid $G(X,\NN,\sigma)$ associated to a local homeomorphism $\sigma:X\rightarrow X$. In this case the ambient group is $\ZZ$ and it is well-known, that every $2$-cocycle on $\ZZ$ is a coboundary, and hence the cohomology group $H^2(\ZZ,\TT)$ is trivial. On the other hand it was observed in \cite{zbMATH01752404} that the group of (isomorphism classes of) twists $Tw(G(X,\NN,\sigma))$ is isomorphic to the sheaf cohomology group $H^1(X,\mathcal{S})$ of $X$, where $\mathcal{S}$ is the sheaf of germs of continuous $\TT$-valued functions on $X$.
	If $X$ is zero-dimensional, then $H^1(X,\mathcal{S})$ vanishes and hence ordinary Renault-Deaconu groupoids (i.e. the case where $P=\NN)$ over totally disconnected spaces do not admit any non-trivial twists.
	
	This begs the following question: If $(X,P,T)$ is a directed action of $P\subseteq \Gamma$ on a locally compact totally disconnected Hausdorff space $X$, is every twist on $G(X,P,T)$ induced by a $2$-cocycle on $\Gamma$?

	\subsection{$P$-graphs}
	Interesting examples of the semidirect product group\-oids studied above arise from higher-rank graphs or more general $P$-graphs as introduced in \cite{MR3592511,MR3871826}. Let $P$ be a subsemigroup of a countable discrete group $\Gamma$. According to \cite[Definition~6.1]{MR3592511} a (discrete) $P$-graph $\Lambda$ is a small category $\Lambda$ endowed with a \textit{degree} functor $d:\Lambda\rightarrow P$ satisfying the unique factorisation property.
	
	Associated with a $P$-graph $\Lambda$ is a canonical action $T_\Lambda$ of $P$ on the infinite path space $X_\Lambda$, which in turn gives rise to the ample groupoid $G_\Lambda:=G(X_\Lambda,P,T_\Lambda)$.
	
	Let us specialise to the case where $P= \NN^k\times F$ for some countable abelian group $F$.
	By \cite[Proposition~5.7, Corollary~5.10]{MR3871826} every $2$-cocycle $c$ on $\Lambda$ gives rise to a continuous $2$-cocycle $\omega_c$ and hence a twist $\Sigma_c$ on $G_\Lambda$ such that
	$C^*(\Lambda,c)$ is canonically isomorphic to $C^*(G_\Lambda;\Sigma_c)$.
	
	A homotopy of $2$-cocycles for a $P$-graph $\Lambda$ is a family $(c_t)_{t\in [0,1]}$ of $2$-cocycles, such that for each pair $(\lambda,\mu)\in \Lambda^{\ast 2}$ the map $t\mapsto c_t(\lambda,\mu)\in\TT$ is continuous. Following the line of argument in \cite[Proposition~3.8]{MR3407179} every homotopy of $2$-cocycles on $\Lambda$ gives rise to a continuous homotopy of twists over $G_\Lambda\times [0,1]$.
	\begin{kor}
		Let $P= \NN^k\times F$ for some countable abelian group $F$. Let $\Lambda$ be a row-finite $P$-graph with no sources and let $(c_t)_{[t\in[0,1]}$ be a homotopy of $2$-cocycles on $\Lambda$. Then $$\K_*(C^*(\Lambda,c_0))\cong \K_*(C^*(\Lambda,c_1)).$$
	\end{kor}
	\begin{proof}
		Since $\ZZ^k\times F$ is amenable, so is $G_\Lambda$ by \cite[Corollary~6.19]{MR3592511}. Thus, Theorem \ref{Thm:Evalutation Induces Isomorphism on K-theory} implies that 
		$$\K_*(C^*(\Lambda,c_0))\cong\K_*(C_r^*(G_\Lambda;\Sigma_{c_0}))\cong \K_*(C_r^*(G_\Lambda;\Sigma_{c_1}))\cong \K_*(C^*(\Lambda,c_1)).$$
	\end{proof}
	We expect this result to hold in much greater generality. In fact we believe it is true for all $(r,d)$-proper $P$-graphs $\Lambda$ (in the sense of \cite{MR3592511}), where $P$ is a quasi-lattice ordered subsemigroup of a countable discrete amenable group $\Gamma$. This would require a generalisation of several results in \cite{MR3871826} to more general $P$-graphs, which we did not seriously attempt.

\ \newline
{\bf Acknowledgments}. The content of this paper is based on some of the findings from 
the author's doctoral thesis. He would like to thank his supervisor Siegfried Echterhoff for his support and advice and Stuart White for commenting on an early draft of this manuscript. He is further indebted to the anonymous referee for useful comments and remarks.

\bibliographystyle{plain}

\end{document}